\documentclass{amsart}

\usepackage{amssymb,amsmath,amsthm,latexsym,booktabs,todonotes}

\theoremstyle{definition}
\newtheorem{definition}{Definition}

\newtheorem{algorithm}[definition]{Algorithm}
\newtheorem{remark}[definition]{Remark}

\theoremstyle{plain}
\newtheorem{lemma}[definition]{Lemma}

\begin{document}

\title{Ranks of 0-1 arrays of size $2 \times 2 \times 2$ and $2 \times 2 \times 2 \times 2$}

\author[Bremner]{Murray R. Bremner}

\address{Department of Mathematics and Statistics, University of Saskatchewan, Canada}

\email{bremner@math.usask.ca}

\author[Stavrou]{Stavros G. Stavrou}

\address{Department of Mathematics and Statistics, University of Saskatchewan, Canada}

\email{sgs715@mail.usask.ca}



\begin{abstract}
We use computer algebra to determine the ranks of arrays of size $2 \times 2 \times 2$ and $2 \times 2 \times 2 \times 2$
with entries in the set $\{0,1\}$ regarded as a field with two elements, as a Boolean algebra, and as non-negative integers.
In the field case we also determine the canonical forms of the arrays with respect to the action of the direct product of
the general linear groups.
\end{abstract}

\maketitle

%%%%%%%%%%%%%%%%%%%%%%%%%%%%%%%%%%%%%%%%%%%%%%%%%%%%%%%%%%%%%%%%%%%%%%%%

\section{Introduction}

Multidimensional arrays and the related topic of hyperdeterminants have
connections with algebraic geometry and representation theory, and applications
in numerical analysis, signal processing, chemometrics and psychometrics.
For the connections with algebraic geometry see \cite{GKZpaper,GKZbook,Landsberg};
for a connection with representation theory see \cite{Bremner}.
For surveys of the applications, see \cite{Cichocki,KoldaBader,Kroonenberg,Smilde}.

Most research on this topic assumes that the arrays have entries in $\mathbb{R}$ or $\mathbb{C}$,
the fields of real and complex numbers.
In this paper, we consider arrays with entries in $\{0,1\}$, which
can be regarded as the field $\mathbb{F}_2$ with two elements ($1+1=0$), as a Boolean algebra ($1+1=1$),
or as non-negative integers ($1+1=2$).
In this case, the total number of arrays is finite, and the problem of determining
the rank of an array reduces to combinatorial enumeration which can be performed by computer.
We obtain a classification by rank of all $2 \times 2 \times 2$ and
$2 \times 2 \times 2 \times 2$ arrays in these three cases.
In the field case we determine the canonical forms of the arrays with respect to
the action of the finite groups $GL_2(\mathbb{F}_2)^3$ and $GL_2(\mathbb{F}_2)^4$
and the extended groups $GL_2(\mathbb{F}_2)^3 \rtimes S_3$ and $GL_2(\mathbb{F}_2)^4 \rtimes S_4$.

The canonical forms of $2 \times 2 \times 2$ arrays have been determined over the real numbers
\cite{deSilvaLim} and over the complex numbers \cite{Ehrenborg,GKZpaper}.
Analogous results for $2 \times 2 \times 2 \times 2$ arrays over $\mathbb{R}$ or $\mathbb{C}$ have not yet
been found, but see \cite{Huggins}.
Since the results for $2 \times 2 \times 2$ arrays over $\mathbb{R}$ and $\mathbb{C}$ are similar to
the results in this paper over $\mathbb{F}_2$, we hope that our results for $2 \times 2 \times 2 \times 2$
arrays over $\mathbb{F}_2$ will provide some useful information towards a classification of canonical forms of
$2 \times 2 \times 2 \times 2$ arrays over $\mathbb{R}$ and $\mathbb{C}$.

%%%%%%%%%%%%%%%%%%%%%%%%%%%%%%%%%%%%%%%%%%%%%%%%%%%%%%%%%%%%%%%%%%%%%%%%

\section{Preliminaries}

We consider $n$-dimensional arrays of size $2 \times \cdots \times 2$ ($n$ factors, $n = 3, 4$) with entries in $\{0,1\}$.

\begin{definition}
The \textbf{flattening} of the array $X = ( x_{i_1 \cdots i_n} )$, $1 \le i_1, \dots, i_n \le 2$, is
  \[
  \mathrm{flat}(X) =
  [
  \begin{array}{cccccccc}
  x_{1 \cdots 1} & \cdots & x_{i_1 \cdots i_n} & \cdots & x_{2 \cdots 2}
  \end{array}
  ],
  \]
where the entries are in lexicographical order of the $n$-tuples of subscripts:
$i_1 \cdots i_n$ precedes $i'_1 \cdots i'_n$ if and only if $i_j < i'_j$ where $j$ is the least index
with $i_j \ne i'_j$.
\end{definition}

\begin{definition} \label{orderdefinition}
If $X$ and $Y$ are two arrays then $X$ \textbf{precedes} $Y$ if
$\mathrm{flat}(X)$ precedes $\mathrm{flat}(Y)$ in lexicographical order:
that is, $x_{i_1 \cdots i_n} < y_{i_1 \cdots i_n}$ where $i_1 \cdots i_n$ is the least $n$-tuple
with $x_{i_1 \cdots i_n} \ne y_{i_1 \cdots i_n}$.
The \textbf{minimal element} of a set of arrays is defined with respect to this total order.
\end{definition}

\begin{definition}
There are three nonzero 2-dimensional vectors: $[0,1]$, $[1,0]$, $[1,1]$.
The \textbf{outer product} $X = V_1 \otimes \cdots \otimes V_n$ of $n$ vectors $V_j = [ v_{j1}, v_{j2} ]$,
$1 \le j \le n$, is
  \[
  X = ( x_{i_1 \cdots i_n} ),
  \qquad
  x_{i_1 \cdots i_n} = v_{1 i_1} \cdots v_{n i_n}.
  \]
\end{definition}

\begin{definition}
The \textbf{rank} of $X = ( x_{i_1 \cdots i_n} )$ is the minimal number $R$ of terms
in the expression of $X$ as a sum of outer products:
  \[
  X = \sum_{r=1}^R V_1^{(r)} \otimes \cdots \otimes V_n^{(r)}.
  \]
The definition of addition depends on the algebraic structure of $\{0,1\}$.  For the field
with two elements, $1+1 = 0$; for the Boolean algebra, $1+1 = 1$; for non-negative integers,
$1+1 = 2$, and this means that we exclude any sums in which two outer products both have an
entry 1 in the same position.
\end{definition}

\begin{lemma} \label{orderlemma}
The only array of rank 0 is the zero array.
An array of rank 1 is the same as an outer product of nonzero vectors,
and there are $3^n$ such arrays.
The total number of $n$-dimensional $2 \times \cdots \times 2$ arrays with entries in $\{0,1\}$ is $2^{2^n}$.
\end{lemma}

\begin{algorithm} \label{basicalgorithm}
Fix a dimension $n$.
Assume that we have already computed the arrays of rank $r$.
To compute the arrays of rank $r+1$, we consider all sums $X + Y$ where $\mathrm{rank}(X) = r$
and $\mathrm{rank}(Y) = 1$.  Clearly $\mathrm{rank}(X+Y) \le r+1$, but it is possible that
$\mathrm{rank}(X+Y) \le r$,
so we only retain those $X+Y$ which have not already been computed: the arrays which have rank exactly $r+1$.
This algorithm is presented in pseudocode for $n = 3$ in Table \ref{basicalgorithmtable}.
\end{algorithm}

\begin{table}

\begin{itemize}
\item[]
\texttt{flatten}$( x )$
\item[] \qquad
\texttt{return}$( [ \, x_{111}, \, x_{112}, \, x_{121}, \, x_{122}, \, x_{211}, \, x_{212}, \, x_{221}, \, x_{222} \, ] )$
\end{itemize}

\bigskip

\begin{itemize}
\item[]
\texttt{outerproduct}$( a, b, c )$
\item[] \qquad
for $i = 1, 2$ do for $j = 1, 2$ do for $k = 1, 2$ do:
\item[] \qquad \qquad
set $x_{ijk} \leftarrow a_i b_j c_k$
\item[] \qquad
\texttt{return}( $x$ )
\end{itemize}

\bigskip

\begin{itemize}
\item
set $\texttt{vectors} \leftarrow \{ \, [1,0], \, [0,1], \, [1,1] \, \}$
\item
set \texttt{arrayset}$[0] \leftarrow \{ [0,0,0,0,0,0,0,0] \}$
\item
set \texttt{arrayset}$[1] \leftarrow \{ \, \}$
\item
for $a$ in \texttt{vectors} do for $b$ in \texttt{vectors} do for $c$ in \texttt{vectors} do
  \begin{itemize}
  \item
  set $x \leftarrow \texttt{flatten}( \texttt{outerproduct}( a, b, c ) )$
  \item
  if $x \notin \texttt{arrayset}[0]$ and $x \notin \texttt{arrayset}[1]$ then
  \item[] \quad
  set $\texttt{arrayset}[1] \leftarrow \texttt{arrayset}[1] \cup \{ x \}$
  \end{itemize}
\item
set $r \leftarrow 1$
\item
while $\texttt{arrayset}[r] \ne \{ \, \}$ do:
  \begin{itemize}
  \item
  set $\texttt{arrayset}[r{+}1] \leftarrow \{ \, \}$
  \item
  for $x \in \texttt{arrayset}[r]$ do for $y \in \texttt{arrayset}[1]$ do
    \begin{itemize}
    \item
    set $z \leftarrow [ \, x_1{+}y_1, \, \dots, \, x_8{+}y_8 \, ]$
    \item
    if $z \notin \texttt{arrayset}[s]$ for $s = 0, \dots, r{+}1$ then
    \item[] \quad
      set $\texttt{arrayset}[r{+}1] \leftarrow \texttt{arrayset}[r{+}1] \cup \{ z \}$
    \end{itemize}
    \item
    set $r \leftarrow r + 1$
  \end{itemize}
\item
set $\texttt{maximumrank} \leftarrow r - 1$
\end{itemize}

\bigskip

\caption{Algorithm \ref{basicalgorithm} in pseudocode}
\label{basicalgorithmtable}
\end{table}

If $\{0,1\}$ is the field $\mathbb{F}_2$ with two elements then we consider the action of
the direct product of $n$ copies of the general linear group $GL_2(\mathbb{F}_2)$.

\begin{definition}
The group $GL_2(\mathbb{F}_2)$ consists of six matrices in lexicographical order:
  \[
  \left[ \begin {array}{cc} 0 & 1 \\ 1 & 0 \end{array} \right], \quad
  \left[ \begin {array}{cc} 0 & 1 \\ 1 & 1 \end{array} \right], \quad
  \left[ \begin {array}{cc} 1 & 0 \\ 0 & 1 \end{array} \right], \quad
  \left[ \begin {array}{cc} 1 & 0 \\ 1 & 1 \end{array} \right], \quad
  \left[ \begin {array}{cc} 1 & 1 \\ 0 & 1 \end{array} \right], \quad
  \left[ \begin {array}{cc} 1 & 1 \\ 1 & 0 \end{array} \right].
  \]
\end{definition}

\begin{lemma}
The group $GL_2(\mathbb{F}_2)$ is isomorphic to $S_3$ permuting the nonzero vectors.
\end{lemma}

\begin{definition} \label{groupdefinition}
The \textbf{small symmetry group} of $2 \times \cdots \times 2$ arrays
over $\mathbb{F}_2$ is the direct product $GL_2(\mathbb{F}_2)^n$ acting by
simultaneous changes of basis along the $n$ directions.
The \textbf{large symmetry group} of these arrays is the semi-direct product
$GL_2(\mathbb{F}_2)^n \rtimes S_n$ where $S_n$ acts by permuting the $n$ copies of
$GL_2(\mathbb{F}_2)$.
The element $A \in GL_2(\mathbb{F}_2)$
acts along the first direction of an array $X = ( x_{i_1 \cdots i_n} )$ as follows:
for each $(n{-}1)$-tuple $i_2 \cdots i_n$ we consider the column vector
  \[
  V_{i_2 \cdots i_n} = [ x_{1 i_2 \cdots i_n}, x_{2 i_2 \cdots i_n} ]^t \in \mathbb{F}_2^2,
  \]
and compute $A V_{i_2 \cdots i_n} = [ y_{1 i_2 \cdots i_n}, y_{2 i_2 \cdots i_n} ]^t \in \mathbb{F}_2^2$;
then we define the array $A \cdot X$ to be $Y = ( y_{i_1 \cdots i_n} )$.
The actions along the other $n-1$ directions are similar.
\end{definition}

\begin{lemma}
The actions of the symmetry groups do not change the rank.
\end{lemma}

\begin{proof}
de Silva and Lim \cite[Lemma 2.3, page 1092]{deSilvaLim} applies to any field.
\end{proof}

The actions of the symmetry groups decompose the set of $n$-dimensional arrays
into a disjoint union of orbits; the arrays in each orbit are equivalent under the group action.

\begin{algorithm} \label{groupalgorithm}
Fix a dimension $n$ and consider $2 \times \cdots \times 2$ arrays over $\mathbb{F}_2$.
Assume we have computed the set of arrays for each possible rank,
and that these sets are totally ordered.
For each rank, we perform the following iteration:
  \begin{itemize}
  \item Choose the minimal element of the set of arrays.
  \item Compute the orbit of this element under the action of the symmetry group.
  \item Remove the elements of this orbit from the set of arrays of the given rank.
  \end{itemize}
This iteration terminates when there are no more arrays of the given rank.
This algorithm is presented in pseudocode for $n = 3$ in Table \ref{groupalgorithmtable}.
\end{algorithm}

\begin{table}

\begin{itemize}
\item[]
\texttt{unflatten}$( x )$
\item[] \qquad
set $t \leftarrow 0$
\item[] \qquad
for $i = 1, 2$ do for $j = 1, 2$ do for $k = 1, 2$ do:
\item[] \qquad \qquad
set $t \leftarrow t + 1$
\item[] \qquad \qquad
set $y_{ijk} \leftarrow x_t$
\item[] \qquad
$\texttt{return}( y )$
\end{itemize}

\bigskip

\begin{itemize}
\item[]
\texttt{groupaction}$( g, x, m )$
  \begin{itemize}
  \item[]
  set $y \leftarrow \texttt{unflatten}( x )$
  \item[]
  if $m = 1$ then
  \item[] \qquad
  for $j = 1,2$ do for $k = 1,2$ do
  \item[] \qquad \qquad
  set $v \leftarrow [ \, y_{1jk}, \, y_{2jk} \, ]$
  \item[] \qquad \qquad
  set $w \leftarrow [ \, g_{11} v_1 + g_{12} v_2, \, g_{21} v_1 + g_{22} v_2 \, ]$
  \item[] \qquad \qquad
  for $i = 1, 2$ do: set $y_{ijk} \leftarrow w_i$
  \item[]
  if $m = 2$ then \dots \emph{(similar for second subscript)}
  \item[]
  if $m = 3$ then \dots \emph{(similar for third subscript)}
  \item[]
  \texttt{return}( \texttt{flatten}( $y$ ) )
  \end{itemize}
\end{itemize}

\bigskip

\begin{itemize}
\item[]
\texttt{smallorbit}$( x )$
  \begin{itemize}
  \item[]
  set $\texttt{result} \leftarrow \{\,\}$
  \item[]
  for $a \in GL_2(\mathbb{F}_2)$ do:
  \item[] \qquad
  set $y \leftarrow \texttt{groupaction}( a, x, 1 )$
  \item[] \qquad
  for $b \in GL_2(\mathbb{F}_2)$ do:
  \item[] \qquad \qquad
  set $z \leftarrow \texttt{groupaction}( b, y, 2 )$
  \item[] \qquad \qquad
  for $c \in GL_2(\mathbb{F}_2)$ do:
  \item[] \qquad \qquad \qquad
  set $w \leftarrow \texttt{groupaction}( c, z, 3 )$
  \item[] \qquad \qquad \qquad
  set $\texttt{result} \leftarrow \texttt{result} \cup \{ w \}$
  \item[]
  \texttt{return}( \texttt{result} )
  \end{itemize}
\end{itemize}

\bigskip

\begin{itemize}
\item[]
\texttt{largeorbit}$( x )$
  \begin{itemize}
  \item[]
  set $y \leftarrow \texttt{unflatten}( x )$
  \item[]
  set $\texttt{result} \leftarrow \{\,\}$
  \item[]
  for $p \in S_3$ do:
  \item[] \qquad
  for $i = 1,2$ do for $j = 1,2$ do for $k = 1,2$ do:
  \item[] \qquad \qquad
  set $m \leftarrow [i,j,k]$
  \item[] \qquad \qquad
  set $z_{ijk} \leftarrow y_{m_{p(1)}m_{p(2)}m_{p(3)}}$
  \item[] \qquad
  set $\texttt{result} \leftarrow \texttt{result} \cup \texttt{smallorbit}( \texttt{flatten}( z ) )$
  \item[]
  \texttt{return}( \texttt{result} )
  \end{itemize}
\end{itemize}

\bigskip

\begin{itemize}
\item
for $r = 0,\dots,\texttt{maximumrank}$ do:
\item[] \qquad
set $\texttt{representatives}[r] \leftarrow \{\,\}$
\item[] \qquad
set $\texttt{remaining} \leftarrow \texttt{arrayset}[r]$
\item[] \qquad
while $\texttt{remaining} \ne \{\,\}$ do:
\item[] \qquad \qquad
set $x \leftarrow \texttt{remaining}[1]$
\item[] \qquad \qquad
set $\texttt{xorbit} \leftarrow \texttt{largeorbit}( x )$
\item[] \qquad \qquad
append $\texttt{xorbit}[1]$ to $\texttt{representatives}[r]$
\item[] \qquad \qquad
set $\texttt{remaining} \leftarrow \texttt{remaining} \setminus \texttt{xorbit}$
\end{itemize}

\bigskip

\caption{Algorithm \ref{groupalgorithm} in pseudocode}
\label{groupalgorithmtable}
\end{table}

\begin{definition}
The minimal element in each orbit is the \textbf{canonical form} of the arrays in that orbit.
\end{definition}

\begin{lemma}
Lower bounds for the number of canonical forms for the small symmetry group and the large symmetry group
are respectively
  \[
  \left\lceil
  \frac{2^{2^n}}{6^n}
  \right\rceil,
  \qquad
  \left\lceil
  \frac{2^{2^n}}{6^n n!}
  \right\rceil.
  \]
\end{lemma}

\begin{proof}
Lemma \ref{orderlemma} shows that there are $2^{2^n}$ such arrays.
Definition \ref{groupdefinition} implies that the small and large symmetry groups have orders $6^n$
and $6^n n!$ respectively.
The claim follows from the theory of group actions on a finite set.
\end{proof}

\begin{remark}
For $3 \le n \le 6$, we have the following lower bounds for the number of orbits for the small and large symmetry groups.
From this it is clear that complete results will not be publishable for $n \ge 5$:
  \begin{center}
  \begin{tabular}{lrr}
  $n$ &\qquad lower bound (small group) &\qquad lower bound (large group) \\
  3 &\qquad 2 &\qquad 1 \\
  4 &\qquad 51 &\qquad 3 \\
  5 &\qquad 552337 &\qquad 4603 \\
  6 &\qquad 395377745064077 &\qquad 549135757034
  \end{tabular}
  \end{center}
\end{remark}

%%%%%%%%%%%%%%%%%%%%%%%%%%%%%%%%%%%%%%%%%%%%%%%%%%%%%%%%%%%%%%%%%%%%%%%%

  \begin{table}
  \[
  \begin{array}{ccc}
  \text{rank} & \text{orbit size} & \text{canonical form}
  \\
  \toprule
  0 &   1 & \left[ \begin{array}{cc|cc} 0 & 0 & 0 & 0 \\ 0 & 0 & 0 & 0 \end{array} \right]
  \\
  \midrule
  1 &  27 & \left[ \begin{array}{cc|cc} 0 & 0 & 0 & 0 \\ 0 & 0 & 0 & 1 \end{array} \right]
  \\
  \midrule
  2 &  54 & \left[ \begin{array}{cc|cc} 0 & 0 & 0 & 1 \\ 0 & 0 & 1 & 0 \end{array} \right]
  \\[8pt]
  2 & 108 & \left[ \begin{array}{cc|cc} 0 & 0 & 1 & 0 \\ 0 & 1 & 0 & 0 \end{array} \right]
  \\
  \midrule
  3 &  54 & \left[ \begin{array}{cc|cc} 0 & 0 & 0 & 1 \\ 0 & 1 & 1 & 0 \end{array} \right]
  \\[8pt]
  3 &  12 & \left[ \begin{array}{cc|cc} 0 & 1 & 1 & 0 \\ 1 & 0 & 1 & 1 \end{array} \right]
  \\
  \bottomrule
  \end{array}
  \]
  \medskip
  \caption{Large orbits of $2 \times 2 \times 2$ arrays over $\mathbb{F}_2$}
  \label{table222modular}
  \end{table}

  \begin{table}
  \[
  \begin{array}{cccc}
  \text{rank} & \text{ones} & \text{number} & \text{representative}
  \\
  \toprule
  0 &    0 &   1 & \left[ \begin{array}{cc|cc} 0 & 0 & 0 & 0 \\ 0 & 0 & 0 & 0 \end{array} \right] \\ \midrule
  1 &    1 &   8 & \left[ \begin{array}{cc|cc} 0 & 0 & 0 & 0 \\ 0 & 0 & 0 & 1 \end{array} \right] \\[8pt]
  1 &    2 &  12 & \left[ \begin{array}{cc|cc} 0 & 0 & 0 & 0 \\ 0 & 0 & 1 & 1 \end{array} \right] \\[8pt]
  1 &    4 &   6 & \left[ \begin{array}{cc|cc} 0 & 0 & 1 & 1 \\ 0 & 0 & 1 & 1 \end{array} \right] \\[8pt]
  1 &    8 &   1 & \left[ \begin{array}{cc|cc} 1 & 1 & 1 & 1 \\ 1 & 1 & 1 & 1 \end{array} \right] \\ \midrule
  2 &    2 &  16 & \left[ \begin{array}{cc|cc} 0 & 0 & 0 & 1 \\ 0 & 0 & 1 & 0 \end{array} \right] \\[8pt]
  2 &    3 &  48 & \left[ \begin{array}{cc|cc} 0 & 0 & 0 & 1 \\ 0 & 0 & 1 & 1 \end{array} \right] \\[8pt]
  2 &    4 &  30 & \left[ \begin{array}{cc|cc} 0 & 0 & 1 & 0 \\ 0 & 1 & 1 & 1 \end{array} \right] \\[8pt]
  2 &    5 &  24 & \left[ \begin{array}{cc|cc} 0 & 0 & 1 & 1 \\ 0 & 1 & 1 & 1 \end{array} \right] \\[8pt]
  2 &    6 &  12 & \left[ \begin{array}{cc|cc} 0 & 0 & 1 & 1 \\ 1 & 1 & 1 & 1 \end{array} \right] \\ \midrule
  3 &    3 &   8 & \left[ \begin{array}{cc|cc} 0 & 0 & 0 & 1 \\ 0 & 1 & 1 & 0 \end{array} \right] \\[8pt]
  3 &    4 &  32 & \left[ \begin{array}{cc|cc} 0 & 0 & 0 & 1 \\ 0 & 1 & 1 & 1 \end{array} \right] \\[8pt]
  3 &    5 &  24 & \left[ \begin{array}{cc|cc} 0 & 0 & 1 & 1 \\ 1 & 1 & 0 & 1 \end{array} \right] \\[8pt]
  3 &    6 &  16 & \left[ \begin{array}{cc|cc} 0 & 1 & 1 & 1 \\ 1 & 0 & 1 & 1 \end{array} \right] \\[8pt]
  3 &    7 &   8 & \left[ \begin{array}{cc|cc} 0 & 1 & 1 & 1 \\ 1 & 1 & 1 & 1 \end{array} \right] \\ \midrule
  4 &    4 &   2 & \left[ \begin{array}{cc|cc} 0 & 1 & 1 & 0 \\ 1 & 0 & 0 & 1 \end{array} \right] \\[8pt]
  4 &    5 &   8 & \left[ \begin{array}{cc|cc} 0 & 1 & 1 & 0 \\ 1 & 0 & 1 & 1 \end{array} \right] \\
  \bottomrule
  \end{array}
  \]
  \medskip
  \caption{Ranks and minimal representatives for $2 \times 2 \times 2$ Boolean arrays}
  \label{table222boolean}
  \end{table}

\section{Arrays of size $2 \times 2 \times 2$}

The set of $2 \times 2 \times 2$ arrays with entries in $\{0,1\}$ contains 256 elements.
We represent such an array $X = ( x_{ijk} )$ in the matrix form
  \[
  \mathrm{Mat}(X) =
  \left[
  \begin{array}{cc|cc}
  x_{111} & x_{121} & x_{112} & x_{122} \\
  x_{211} & x_{221} & x_{212} & x_{222}
  \end{array}
  \right],
  \]
where the third subscript distinguishes the left and right blocks, which are the first and second
frontal slices.

\subsection{The field with two elements ($1+1=0$)}

Algorithm \ref{basicalgorithm} shows that in this case the maximum rank is 3; the number of arrays
of each rank is 1, 27, 162, 66.
The percentages of ranks 0, 1, 2, 3 are approximately 0, 11, 63, 26; in contrast, the percentages
over $\mathbb{R}$ are approximately 0, 0, 79, 21.
For the large symmetry group $GL_2(\mathbb{F}_2)^3 \rtimes S_3$, the ranks, orbit sizes, and canonical forms
are given in Table \ref{table222modular}.
For the small symmetry group $GL_2(\mathbb{F}_2)^3$, the first orbit in rank 2 splits into three orbits
each of size 18 with canonical forms
  \[
  \left[ \begin{array}{cc|cc} 0 & 0 & 0 & 1 \\ 0 & 0 & 1 & 0 \end{array} \right],
  \qquad
  \left[ \begin{array}{cc|cc} 0 & 0 & 0 & 0 \\ 0 & 1 & 1 & 0 \end{array} \right],
  \qquad
  \left[ \begin{array}{cc|cc} 0 & 0 & 0 & 1 \\ 0 & 1 & 0 & 0 \end{array} \right].
  \]
For the small symmetry group, there are eight orbits, the same as in the real case.

\subsection{The Boolean algebra ($1+1=1$)}

The maximum rank is 4; the number of arrays of each rank is 1, 27, 130, 88, 10.
The percentages of ranks 0, 1, 2, 3, 4 are approximately 0, 11, 51, 34, 4.
Instead of canonical forms for a group action, which do not exist in the Boolean case, we partition the arrays in each rank
by the number of entries equal to 1; the results are given in Table \ref{table222boolean}.

\subsection{Non-negative integers ($1+1=2$)}

The results are the same as in the Boolean case.
This has the corollary that every $2 \times 2 \times 2$ Boolean array of rank $r$
can be written as the sum of $r$ outer products such that no two terms have an entry 1 in the same position.
(When we consider arrays of size $2 \times 2 \times 2 \times 2$, the Boolean and integer cases
will no longer be identical.)

%%%%%%%%%%%%%%%%%%%%%%%%%%%%%%%%%%%%%%%%%%%%%%%%%%%%%%%%%%%%%%%%%%%%%%%%

\section{Arrays of size $2 \times 2 \times 2 \times 2$}

The set of $2 \times 2 \times 2 \times 2$ arrays with entries in $\{0,1\}$ contains 65536 elements.

\subsection{The field with two elements ($1+1=0$)}

Algorithm \ref{basicalgorithm} shows that in this case the maximum rank is 6.
The number of arrays of each rank and the approximate percentages are as follows:
  \[
  \begin{array}{lrrrrrrr}
  \text{rank} &\; 0 &\; 1 &\; 2 &\; 3 &\; 4 &\; 5 &\; 6 \\
  \text{number} &\; 1 &\; 81 &\; 2268 &\; 21744 &\; 37530 &\; 3888 &\; 24 \\
  \approx\% &\; 0.002 &\; 0.124 &\; 3.461 &\; 33.179 &\; 57.266 &\; 5.933 &\; 0.037
  \end{array}
  \]
For the large symmetry group $GL_2(\mathbb{F}_2)^4 \rtimes S_4$, there are 30 orbits;
the ranks, orbit sizes, and canonical forms are given in Table \ref{table2222modular}.
For the small symmetry group $GL_2(\mathbb{F}_2)^4$, there are 112 orbits.
The large orbits split into small orbits as follows, where we mention only those
large orbits that are not small orbits, and write $x \rightarrow y \cdot z$
to indicate that large orbit $x$ splits into $y$ small orbits each of size $z$:
  \begin{center}
  \begin{tabular}{llll}
  rank 2 &\quad rank 3 &\quad rank 4 &\quad rank 5 \\
  \midrule
  $ 3 \rightarrow 6 \cdot   54$ &\quad $ 6 \rightarrow 4 \cdot  162$ &\quad $15 \rightarrow  4 \cdot  648$ &\quad $26 \rightarrow 6 \cdot 108$ \\
  $ 4 \rightarrow 4 \cdot  324$ &\quad $ 7 \rightarrow 4 \cdot   36$ &\quad $16 \rightarrow  4 \cdot 1296$ \\
                                &\quad $ 8 \rightarrow 6 \cdot  648$ &\quad $17 \rightarrow  3 \cdot   36$ \\
                                &\quad $ 9 \rightarrow 4 \cdot  648$ &\quad $18 \rightarrow  3 \cdot  324$ \\
                                &\quad $10 \rightarrow 4 \cdot  648$ &\quad $19 \rightarrow  6 \cdot  324$ \\
                                &\quad $11 \rightarrow 3 \cdot 1296$ &\quad $20 \rightarrow  3 \cdot  648$ \\
                                &\quad $12 \rightarrow 6 \cdot 1296$ &\quad $21 \rightarrow 12 \cdot  648$ \\
                                &\quad                               &\quad $22 \rightarrow  6 \cdot  216$ \\
                                &\quad                               &\quad $23 \rightarrow  6 \cdot 1296$ \\
                                &\quad                               &\quad $24 \rightarrow  3 \cdot 1296$ \\
                                &\quad                               &\quad $25 \rightarrow  6 \cdot  648$ \\
  \midrule
  \end{tabular}
  \end{center}

\subsection{The Boolean algebra ($1+1=1$)}

The maximum rank is 8.  The number of arrays of each rank and the approximate percentages are as follows:
  \[
  \begin{array}{lrrrrrrrrr}
  \text{rank} &\, 0 &\, 1 &\, 2 &\, 3 &\, 4 &\, 5 &\, 6 &\, 7 &\, 8 \\
  \text{number} &\, 1 &\, 81 &\, 1804 &\, 13472 &\, 28904 &\, 17032 &\, 3704 &\, 512 &\, 26 \\
  \approx\% &\, 0.002 &\, 0.124 &\, 2.753 &\, 20.557 &\, 44.104 &\, 25.989 &\, 5.652 &\, 0.781 &\, 0.04
  \end{array}
  \]
The results are given in Table \ref{table2222boolean} by rank and number of entries equal to 1.

\subsection{Non-negative integers ($1+1=2$)}

The maximum rank is 8.
The number of arrays of each rank and the approximate percentages are as follows:
  \[
  \begin{array}{lrrrrrrrrr}
  \text{rank} &\, 0 &\, 1 &\, 2 &\, 3 &\, 4 &\, 5 &\, 6 &\, 7 &\, 8 \\
  \text{number} &\, 1 &\, 81 &\, 1756 &\, 12848 &\, 28788 &\, 17568 &\, 3908 &\, 560 &\, 26 \\
  \approx\% &\, 0.002 &\, 0.124 &\, 2.679 &\, 19.604 &\, 43.927 &\, 26.807 &\, 5.963 &\, 0.854 &\, 0.04
  \end{array}
  \]
The results are given in Table \ref{table2222integer} by rank and number of entries equal to 1.

  \begin{table}
  \[
  \begin{array}{cccc} \\
  & \text{rank} & \text{large orbit size} &\quad \text{canonical form (flattened)} \\
  \toprule
   1 & 0 &     1  &\quad 0 \; 0 \; 0 \; 0 \; 0 \; 0 \; 0 \; 0 \; 0 \; 0 \; 0 \; 0 \; 0 \; 0 \; 0 \; 0 \\
  \midrule
   2 & 1 &    81  &\quad 0 \; 0 \; 0 \; 0 \; 0 \; 0 \; 0 \; 0 \; 0 \; 0 \; 0 \; 0 \; 0 \; 0 \; 0 \; 1 \\
  \midrule
   3 & 2 &   324  &\quad 0 \; 0 \; 0 \; 0 \; 0 \; 0 \; 0 \; 0 \; 0 \; 0 \; 0 \; 0 \; 0 \; 1 \; 1 \; 0 \\
   4 & 2 &  1296  &\quad 0 \; 0 \; 0 \; 0 \; 0 \; 0 \; 0 \; 0 \; 0 \; 0 \; 0 \; 1 \; 1 \; 0 \; 0 \; 0 \\
   5 & 2 &   648  &\quad 0 \; 0 \; 0 \; 0 \; 0 \; 0 \; 0 \; 1 \; 1 \; 0 \; 0 \; 0 \; 0 \; 0 \; 0 \; 0 \\
  \midrule
   6 & 3 &   648  &\quad 0 \; 0 \; 0 \; 0 \; 0 \; 0 \; 0 \; 0 \; 0 \; 0 \; 0 \; 1 \; 0 \; 1 \; 1 \; 0 \\
   7 & 3 &   144  &\quad 0 \; 0 \; 0 \; 0 \; 0 \; 0 \; 0 \; 0 \; 0 \; 1 \; 1 \; 0 \; 1 \; 0 \; 1 \; 1 \\
   8 & 3 &  3888  &\quad 0 \; 0 \; 0 \; 0 \; 0 \; 0 \; 0 \; 1 \; 0 \; 0 \; 0 \; 1 \; 1 \; 0 \; 0 \; 0 \\
   9 & 3 &  2592  &\quad 0 \; 0 \; 0 \; 0 \; 0 \; 0 \; 0 \; 1 \; 0 \; 0 \; 1 \; 0 \; 1 \; 1 \; 0 \; 0 \\
  10 & 3 &  2592  &\quad 0 \; 0 \; 0 \; 0 \; 0 \; 0 \; 0 \; 1 \; 0 \; 1 \; 1 \; 0 \; 1 \; 0 \; 1 \; 0 \\
  11 & 3 &  3888  &\quad 0 \; 0 \; 0 \; 0 \; 0 \; 0 \; 0 \; 1 \; 1 \; 0 \; 0 \; 0 \; 0 \; 0 \; 1 \; 0 \\
  12 & 3 &  7776  &\quad 0 \; 0 \; 0 \; 0 \; 0 \; 0 \; 0 \; 1 \; 1 \; 0 \; 0 \; 0 \; 0 \; 1 \; 1 \; 0 \\
  13 & 3 &   216  &\quad 0 \; 0 \; 0 \; 1 \; 1 \; 0 \; 0 \; 0 \; 1 \; 1 \; 1 \; 0 \; 1 \; 1 \; 1 \; 1 \\
  \midrule
  14 & 4 &   162  &\quad 0 \; 0 \; 0 \; 0 \; 0 \; 0 \; 0 \; 1 \; 0 \; 0 \; 0 \; 1 \; 0 \; 1 \; 1 \; 0 \\
  15 & 4 &  2592  &\quad 0 \; 0 \; 0 \; 0 \; 0 \; 0 \; 0 \; 1 \; 0 \; 1 \; 1 \; 0 \; 1 \; 0 \; 0 \; 0 \\
  16 & 4 &  5184  &\quad 0 \; 0 \; 0 \; 0 \; 0 \; 0 \; 0 \; 1 \; 1 \; 0 \; 0 \; 1 \; 0 \; 1 \; 1 \; 0 \\
  17 & 4 &   108  &\quad 0 \; 0 \; 0 \; 0 \; 0 \; 1 \; 1 \; 0 \; 0 \; 1 \; 1 \; 0 \; 0 \; 0 \; 0 \; 0 \\
  18 & 4 &   972  &\quad 0 \; 0 \; 0 \; 0 \; 0 \; 1 \; 1 \; 0 \; 0 \; 1 \; 1 \; 0 \; 0 \; 0 \; 0 \; 1 \\
  19 & 4 &  1944  &\quad 0 \; 0 \; 0 \; 0 \; 0 \; 1 \; 1 \; 0 \; 0 \; 1 \; 1 \; 0 \; 0 \; 0 \; 1 \; 0 \\
  20 & 4 &  1944  &\quad 0 \; 0 \; 0 \; 0 \; 0 \; 1 \; 1 \; 0 \; 0 \; 1 \; 1 \; 1 \; 0 \; 0 \; 1 \; 0 \\
  21 & 4 &  7776  &\quad 0 \; 0 \; 0 \; 0 \; 0 \; 1 \; 1 \; 0 \; 0 \; 1 \; 1 \; 1 \; 1 \; 0 \; 0 \; 0 \\
  22 & 4 &  1296  &\quad 0 \; 0 \; 0 \; 0 \; 0 \; 1 \; 1 \; 0 \; 1 \; 0 \; 1 \; 1 \; 0 \; 0 \; 0 \; 0 \\
  23 & 4 &  7776  &\quad 0 \; 0 \; 0 \; 0 \; 0 \; 1 \; 1 \; 0 \; 1 \; 0 \; 1 \; 1 \; 0 \; 0 \; 0 \; 1 \\
  24 & 4 &  3888  &\quad 0 \; 0 \; 0 \; 1 \; 0 \; 1 \; 1 \; 0 \; 1 \; 0 \; 0 \; 0 \; 0 \; 0 \; 1 \; 1 \\
  25 & 4 &  3888  &\quad 0 \; 0 \; 0 \; 1 \; 0 \; 1 \; 1 \; 0 \; 1 \; 0 \; 0 \; 0 \; 1 \; 0 \; 1 \; 1 \\
  \midrule
  26 & 5 &   648  &\quad 0 \; 0 \; 0 \; 0 \; 0 \; 1 \; 1 \; 0 \; 0 \; 1 \; 1 \; 0 \; 1 \; 0 \; 1 \; 1 \\
  27 & 5 &   648  &\quad 0 \; 0 \; 0 \; 1 \; 0 \; 1 \; 1 \; 0 \; 0 \; 1 \; 1 \; 0 \; 1 \; 0 \; 0 \; 0 \\
  28 & 5 &  1296  &\quad 0 \; 0 \; 0 \; 1 \; 0 \; 1 \; 1 \; 0 \; 0 \; 1 \; 1 \; 0 \; 1 \; 0 \; 0 \; 1 \\
  29 & 5 &  1296  &\quad 0 \; 0 \; 0 \; 1 \; 0 \; 1 \; 1 \; 0 \; 1 \; 0 \; 0 \; 0 \; 0 \; 0 \; 0 \; 1 \\
  \midrule
  30 & 6 &    24  &\quad 0 \; 1 \; 1 \; 0 \; 1 \; 0 \; 1 \; 1 \; 1 \; 0 \; 1 \; 1 \; 1 \; 1 \; 0 \; 1 \\
  \bottomrule
  \end{array}
  \]
  \medskip
  \caption{Large orbits of $2 \times 2 \times 2 \times 2$ arrays over $\mathbb{F}_2$}
  \label{table2222modular}
  \end{table}

  \begin{table} \tiny
  \[
  \begin{array}{ccccc}
  & \text{rank} & \text{ones} & \text{size} & \text{representative}
  \\
  \toprule
   1 & 0 &  0 &    1 & 0 \, 0 \, 0 \, 0 \, 0 \, 0 \, 0 \, 0 \, 0 \, 0 \, 0 \, 0 \, 0 \, 0 \, 0 \, 0 \\
  \midrule
   2 & 1 &  1 &   16 & 0 \, 0 \, 0 \, 0 \, 0 \, 0 \, 0 \, 0 \, 0 \, 0 \, 0 \, 0 \, 0 \, 0 \, 0 \, 1 \\ %%%
   3 & 1 &  2 &   32 & 0 \, 0 \, 0 \, 0 \, 0 \, 0 \, 0 \, 0 \, 0 \, 0 \, 0 \, 0 \, 0 \, 0 \, 1 \, 1 \\ %%%
   4 & 1 &  4 &   24 & 0 \, 0 \, 0 \, 0 \, 0 \, 0 \, 0 \, 0 \, 0 \, 0 \, 0 \, 0 \, 1 \, 1 \, 1 \, 1 \\ %%%
   5 & 1 &  8 &    8 & 0 \, 0 \, 0 \, 0 \, 0 \, 0 \, 0 \, 0 \, 1 \, 1 \, 1 \, 1 \, 1 \, 1 \, 1 \, 1 \\ %%%
   6 & 1 & 16 &    1 & 1 \, 1 \, 1 \, 1 \, 1 \, 1 \, 1 \, 1 \, 1 \, 1 \, 1 \, 1 \, 1 \, 1 \, 1 \, 1 \\
  \midrule
   7 & 2 &  2 &   88 & 0 \, 0 \, 0 \, 0 \, 0 \, 0 \, 0 \, 0 \, 0 \, 0 \, 0 \, 0 \, 0 \, 1 \, 1 \, 0 \\ %%%
   8 & 2 &  3 &  352 & 0 \, 0 \, 0 \, 0 \, 0 \, 0 \, 0 \, 0 \, 0 \, 0 \, 0 \, 0 \, 0 \, 1 \, 1 \, 1 \\ %%%
   9 & 2 &  4 &  352 & 0 \, 0 \, 0 \, 0 \, 0 \, 0 \, 0 \, 0 \, 0 \, 0 \, 0 \, 1 \, 1 \, 0 \, 1 \, 1 \\ %%%
  10 & 2 &  5 &  288 & 0 \, 0 \, 0 \, 0 \, 0 \, 0 \, 0 \, 0 \, 0 \, 0 \, 0 \, 1 \, 1 \, 1 \, 1 \, 1 \\ %%%
  11 & 2 &  6 &  384 & 0 \, 0 \, 0 \, 0 \, 0 \, 0 \, 0 \, 0 \, 0 \, 0 \, 1 \, 1 \, 1 \, 1 \, 1 \, 1 \\ %%%
  12 & 2 &  7 &   48 & 0 \, 0 \, 0 \, 0 \, 0 \, 0 \, 1 \, 1 \, 0 \, 1 \, 0 \, 1 \, 0 \, 1 \, 1 \, 1 \\ %%%
  13 & 2 &  8 &  108 & 0 \, 0 \, 0 \, 0 \, 0 \, 0 \, 1 \, 1 \, 1 \, 1 \, 0 \, 0 \, 1 \, 1 \, 1 \, 1 \\ %%%
  14 & 2 &  9 &   64 & 0 \, 0 \, 0 \, 0 \, 0 \, 0 \, 0 \, 1 \, 1 \, 1 \, 1 \, 1 \, 1 \, 1 \, 1 \, 1 \\ %%%
  15 & 2 & 10 &   96 & 0 \, 0 \, 0 \, 0 \, 0 \, 0 \, 1 \, 1 \, 1 \, 1 \, 1 \, 1 \, 1 \, 1 \, 1 \, 1 \\ %%%
  16 & 2 & 12 &   24 & 0 \, 0 \, 0 \, 0 \, 1 \, 1 \, 1 \, 1 \, 1 \, 1 \, 1 \, 1 \, 1 \, 1 \, 1 \, 1 \\
  \midrule
  17 & 3 &  3 &  208 & 0 \, 0 \, 0 \, 0 \, 0 \, 0 \, 0 \, 0 \, 0 \, 0 \, 0 \, 1 \, 0 \, 1 \, 1 \, 0 \\ %%%
  18 & 3 &  4 & 1216 & 0 \, 0 \, 0 \, 0 \, 0 \, 0 \, 0 \, 0 \, 0 \, 0 \, 0 \, 1 \, 0 \, 1 \, 1 \, 1 \\ %%%
  19 & 3 &  5 & 2304 & 0 \, 0 \, 0 \, 0 \, 0 \, 0 \, 0 \, 0 \, 0 \, 0 \, 1 \, 1 \, 1 \, 1 \, 0 \, 1 \\ %%%
  20 & 3 &  6 & 2512 & 0 \, 0 \, 0 \, 0 \, 0 \, 0 \, 0 \, 0 \, 0 \, 1 \, 1 \, 0 \, 1 \, 1 \, 1 \, 1 \\ %%%
  21 & 3 &  7 & 2656 & 0 \, 0 \, 0 \, 0 \, 0 \, 0 \, 0 \, 0 \, 0 \, 1 \, 1 \, 1 \, 1 \, 1 \, 1 \, 1 \\ %%%
  22 & 3 &  8 & 1904 & 0 \, 0 \, 0 \, 0 \, 0 \, 0 \, 0 \, 1 \, 1 \, 1 \, 1 \, 0 \, 1 \, 1 \, 1 \, 1 \\ %%%
  23 & 3 &  9 & 1056 & 0 \, 0 \, 0 \, 0 \, 0 \, 0 \, 1 \, 1 \, 1 \, 1 \, 0 \, 1 \, 1 \, 1 \, 1 \, 1 \\ %%%
  24 & 3 & 10 &  656 & 0 \, 0 \, 0 \, 0 \, 0 \, 1 \, 1 \, 0 \, 1 \, 1 \, 1 \, 1 \, 1 \, 1 \, 1 \, 1 \\ %%%
  25 & 3 & 11 &  576 & 0 \, 0 \, 0 \, 0 \, 0 \, 1 \, 1 \, 1 \, 1 \, 1 \, 1 \, 1 \, 1 \, 1 \, 1 \, 1 \\ %%%
  26 & 3 & 12 &  256 & 0 \, 0 \, 0 \, 1 \, 1 \, 0 \, 1 \, 1 \, 1 \, 1 \, 1 \, 1 \, 1 \, 1 \, 1 \, 1 \\ %%%
  27 & 3 & 13 &   96 & 0 \, 0 \, 0 \, 1 \, 1 \, 1 \, 1 \, 1 \, 1 \, 1 \, 1 \, 1 \, 1 \, 1 \, 1 \, 1 \\ %%%
  28 & 3 & 14 &   32 & 0 \, 0 \, 1 \, 1 \, 1 \, 1 \, 1 \, 1 \, 1 \, 1 \, 1 \, 1 \, 1 \, 1 \, 1 \, 1 \\
  \midrule
  29 & 4 &  4 &  228 & 0 \, 0 \, 0 \, 0 \, 0 \, 0 \, 0 \, 0 \, 0 \, 1 \, 1 \, 0 \, 1 \, 0 \, 0 \, 1 \\ %%%
  30 & 4 &  5 & 1648 & 0 \, 0 \, 0 \, 0 \, 0 \, 0 \, 0 \, 0 \, 0 \, 1 \, 1 \, 0 \, 1 \, 0 \, 1 \, 1 \\ %%%
  31 & 4 &  6 & 4048 & 0 \, 0 \, 0 \, 0 \, 0 \, 0 \, 0 \, 1 \, 0 \, 0 \, 1 \, 1 \, 1 \, 1 \, 1 \, 0 \\ %%%
  32 & 4 &  7 & 5856 & 0 \, 0 \, 0 \, 0 \, 0 \, 0 \, 0 \, 1 \, 0 \, 1 \, 1 \, 0 \, 1 \, 1 \, 1 \, 1 \\ %%%
  33 & 4 &  8 & 6304 & 0 \, 0 \, 0 \, 0 \, 0 \, 0 \, 0 \, 1 \, 0 \, 1 \, 1 \, 1 \, 1 \, 1 \, 1 \, 1 \\ %%%
  34 & 4 &  9 & 5200 & 0 \, 0 \, 0 \, 0 \, 0 \, 0 \, 1 \, 1 \, 0 \, 1 \, 1 \, 1 \, 1 \, 1 \, 1 \, 1 \\ %%%
  35 & 4 & 10 & 3200 & 0 \, 0 \, 0 \, 0 \, 0 \, 1 \, 1 \, 1 \, 1 \, 0 \, 1 \, 1 \, 1 \, 1 \, 1 \, 1 \\ %%%
  36 & 4 & 11 & 1408 & 0 \, 0 \, 0 \, 0 \, 1 \, 1 \, 1 \, 1 \, 1 \, 1 \, 1 \, 1 \, 0 \, 1 \, 1 \, 1 \\ %%%
  37 & 4 & 12 &  652 & 0 \, 0 \, 0 \, 1 \, 0 \, 1 \, 1 \, 1 \, 1 \, 1 \, 1 \, 1 \, 1 \, 1 \, 1 \, 1 \\ %%%
  38 & 4 & 13 &  256 & 0 \, 0 \, 1 \, 1 \, 1 \, 1 \, 0 \, 1 \, 1 \, 1 \, 1 \, 1 \, 1 \, 1 \, 1 \, 1 \\ %%%
  39 & 4 & 14 &   88 & 0 \, 1 \, 1 \, 0 \, 1 \, 1 \, 1 \, 1 \, 1 \, 1 \, 1 \, 1 \, 1 \, 1 \, 1 \, 1 \\ %%%
  40 & 4 & 15 &   16 & 0 \, 1 \, 1 \, 1 \, 1 \, 1 \, 1 \, 1 \, 1 \, 1 \, 1 \, 1 \, 1 \, 1 \, 1 \, 1 \\
  \midrule
  41 & 5 &  5 &  128 & 0 \, 0 \, 0 \, 0 \, 0 \, 0 \, 0 \, 1 \, 1 \, 0 \, 0 \, 1 \, 0 \, 1 \, 1 \, 0 \\ %%%
  42 & 5 &  6 & 1008 & 0 \, 0 \, 0 \, 0 \, 0 \, 0 \, 0 \, 1 \, 1 \, 0 \, 0 \, 1 \, 0 \, 1 \, 1 \, 1 \\ %%%
  43 & 5 &  7 & 2416 & 0 \, 0 \, 0 \, 0 \, 0 \, 0 \, 1 \, 1 \, 0 \, 1 \, 1 \, 0 \, 1 \, 1 \, 0 \, 1 \\ %%%
  44 & 5 &  8 & 3568 & 0 \, 0 \, 0 \, 0 \, 0 \, 1 \, 1 \, 0 \, 0 \, 1 \, 1 \, 0 \, 1 \, 1 \, 1 \, 1 \\ %%%
  45 & 5 &  9 & 4016 & 0 \, 0 \, 0 \, 0 \, 0 \, 1 \, 1 \, 0 \, 0 \, 1 \, 1 \, 1 \, 1 \, 1 \, 1 \, 1 \\ %%%
  46 & 5 & 10 & 3088 & 0 \, 0 \, 0 \, 0 \, 0 \, 1 \, 1 \, 1 \, 0 \, 1 \, 1 \, 1 \, 1 \, 1 \, 1 \, 1 \\ %%%
  47 & 5 & 11 & 1888 & 0 \, 0 \, 0 \, 1 \, 0 \, 1 \, 1 \, 1 \, 1 \, 1 \, 1 \, 0 \, 1 \, 1 \, 1 \, 1 \\ %%%
  48 & 5 & 12 &  712 & 0 \, 0 \, 0 \, 1 \, 1 \, 1 \, 1 \, 1 \, 1 \, 1 \, 1 \, 1 \, 0 \, 1 \, 1 \, 1 \\ %%%
  49 & 5 & 13 &  208 & 0 \, 1 \, 1 \, 0 \, 1 \, 0 \, 1 \, 1 \, 1 \, 1 \, 1 \, 1 \, 1 \, 1 \, 1 \, 1 \\
  \midrule
  50 & 6 &  6 &   56 & 0 \, 0 \, 0 \, 0 \, 0 \, 1 \, 1 \, 0 \, 0 \, 1 \, 1 \, 0 \, 1 \, 0 \, 0 \, 1 \\ %%%
  51 & 6 &  7 &  448 & 0 \, 0 \, 0 \, 0 \, 0 \, 1 \, 1 \, 0 \, 0 \, 1 \, 1 \, 0 \, 1 \, 0 \, 1 \, 1 \\ %%%
  52 & 6 &  8 &  848 & 0 \, 0 \, 0 \, 0 \, 0 \, 1 \, 1 \, 1 \, 0 \, 1 \, 1 \, 1 \, 1 \, 0 \, 0 \, 1 \\ %%%
  53 & 6 &  9 &  928 & 0 \, 0 \, 0 \, 1 \, 0 \, 1 \, 1 \, 0 \, 0 \, 1 \, 1 \, 0 \, 1 \, 1 \, 1 \, 1 \\ %%%
  54 & 6 & 10 &  848 & 0 \, 0 \, 0 \, 1 \, 0 \, 1 \, 1 \, 0 \, 0 \, 1 \, 1 \, 1 \, 1 \, 1 \, 1 \, 1 \\ %%%
  55 & 6 & 11 &  416 & 0 \, 0 \, 0 \, 1 \, 0 \, 1 \, 1 \, 1 \, 0 \, 1 \, 1 \, 1 \, 1 \, 1 \, 1 \, 1 \\ %%%
  56 & 6 & 12 &  160 & 0 \, 1 \, 1 \, 0 \, 1 \, 0 \, 1 \, 1 \, 1 \, 1 \, 0 \, 1 \, 1 \, 1 \, 1 \, 1 \\
  \midrule
  57 & 7 &  7 &   16 & 0 \, 0 \, 0 \, 1 \, 0 \, 1 \, 1 \, 0 \, 0 \, 1 \, 1 \, 0 \, 1 \, 0 \, 0 \, 1 \\ %%%
  58 & 7 &  8 &  128 & 0 \, 0 \, 0 \, 1 \, 0 \, 1 \, 1 \, 0 \, 0 \, 1 \, 1 \, 0 \, 1 \, 0 \, 1 \, 1 \\ %%%
  59 & 7 &  9 &  160 & 0 \, 0 \, 0 \, 1 \, 0 \, 1 \, 1 \, 1 \, 1 \, 1 \, 1 \, 0 \, 1 \, 0 \, 0 \, 1 \\ %%%
  60 & 7 & 10 &  112 & 0 \, 0 \, 1 \, 1 \, 1 \, 1 \, 0 \, 1 \, 1 \, 1 \, 0 \, 1 \, 0 \, 1 \, 1 \, 0 \\ %%%
  61 & 7 & 11 &   80 & 0 \, 1 \, 1 \, 0 \, 1 \, 0 \, 0 \, 1 \, 1 \, 0 \, 1 \, 1 \, 1 \, 1 \, 1 \, 1 \\ %%%
  62 & 7 & 12 &   16 & 0 \, 1 \, 1 \, 0 \, 1 \, 0 \, 1 \, 1 \, 1 \, 0 \, 1 \, 1 \, 1 \, 1 \, 1 \, 1 \\
  \midrule
  63 & 8 &  8 &    2 & 0 \, 1 \, 1 \, 0 \, 1 \, 0 \, 0 \, 1 \, 1 \, 0 \, 0 \, 1 \, 0 \, 1 \, 1 \, 0 \\ %%%
  64 & 8 &  9 &   16 & 0 \, 1 \, 1 \, 0 \, 1 \, 0 \, 0 \, 1 \, 1 \, 0 \, 0 \, 1 \, 0 \, 1 \, 1 \, 1 \\ %%%
  65 & 8 & 10 &    8 & 0 \, 1 \, 1 \, 0 \, 1 \, 0 \, 1 \, 1 \, 1 \, 1 \, 0 \, 1 \, 0 \, 1 \, 1 \, 0 \\
  \bottomrule
  \end{array}
  \]
  \medskip
  \caption{Minimal representatives for $2 \times 2 \times 2 \times 2$ Boolean arrays}
  \label{table2222boolean}
  \end{table}

  \begin{table} \tiny
  \[
  \begin{array}{ccccc}
  & \text{rank} & \text{ones} & \text{size} & \text{representative}
  \\
  \toprule
   1 & 0 &  0 &    1 & 0 \, 0 \, 0 \, 0 \, 0 \, 0 \, 0 \, 0 \, 0 \, 0 \, 0 \, 0 \, 0 \, 0 \, 0 \, 0 \\
  \midrule
   2 & 1 &  1 &   16 & 0 \, 0 \, 0 \, 0 \, 0 \, 0 \, 0 \, 0 \, 0 \, 0 \, 0 \, 0 \, 0 \, 0 \, 0 \, 1 \\
   3 & 1 &  2 &   32 & 0 \, 0 \, 0 \, 0 \, 0 \, 0 \, 0 \, 0 \, 0 \, 0 \, 0 \, 0 \, 0 \, 0 \, 1 \, 1 \\
   4 & 1 &  4 &   24 & 0 \, 0 \, 0 \, 0 \, 0 \, 0 \, 0 \, 0 \, 0 \, 0 \, 0 \, 0 \, 1 \, 1 \, 1 \, 1 \\
   5 & 1 &  8 &    8 & 0 \, 0 \, 0 \, 0 \, 0 \, 0 \, 0 \, 0 \, 1 \, 1 \, 1 \, 1 \, 1 \, 1 \, 1 \, 1 \\
   6 & 1 & 16 &    1 & 1 \, 1 \, 1 \, 1 \, 1 \, 1 \, 1 \, 1 \, 1 \, 1 \, 1 \, 1 \, 1 \, 1 \, 1 \, 1 \\
  \midrule
   7 & 2 &  2 &   88 & 0 \, 0 \, 0 \, 0 \, 0 \, 0 \, 0 \, 0 \, 0 \, 0 \, 0 \, 0 \, 0 \, 1 \, 1 \, 0 \\
   8 & 2 &  3 &  352 & 0 \, 0 \, 0 \, 0 \, 0 \, 0 \, 0 \, 0 \, 0 \, 0 \, 0 \, 0 \, 0 \, 1 \, 1 \, 1 \\
   9 & 2 &  4 &  352 & 0 \, 0 \, 0 \, 0 \, 0 \, 0 \, 0 \, 0 \, 0 \, 0 \, 0 \, 1 \, 1 \, 0 \, 1 \, 1 \\
  10 & 2 &  5 &  288 & 0 \, 0 \, 0 \, 0 \, 0 \, 0 \, 0 \, 0 \, 0 \, 0 \, 0 \, 1 \, 1 \, 1 \, 1 \, 1 \\
  11 & 2 &  6 &  384 & 0 \, 0 \, 0 \, 0 \, 0 \, 0 \, 0 \, 0 \, 0 \, 0 \, 1 \, 1 \, 1 \, 1 \, 1 \, 1 \\
  12 & 2 &  8 &  108 & 0 \, 0 \, 0 \, 0 \, 0 \, 0 \, 1 \, 1 \, 1 \, 1 \, 0 \, 0 \, 1 \, 1 \, 1 \, 1 \\
  13 & 2 &  9 &   64 & 0 \, 0 \, 0 \, 0 \, 0 \, 0 \, 0 \, 1 \, 1 \, 1 \, 1 \, 1 \, 1 \, 1 \, 1 \, 1 \\
  14 & 2 & 10 &   96 & 0 \, 0 \, 0 \, 0 \, 0 \, 0 \, 1 \, 1 \, 1 \, 1 \, 1 \, 1 \, 1 \, 1 \, 1 \, 1 \\
  15 & 2 & 12 &   24 & 0 \, 0 \, 0 \, 0 \, 1 \, 1 \, 1 \, 1 \, 1 \, 1 \, 1 \, 1 \, 1 \, 1 \, 1 \, 1 \\
  \midrule
  16 & 3 &  3 &  208 & 0 \, 0 \, 0 \, 0 \, 0 \, 0 \, 0 \, 0 \, 0 \, 0 \, 0 \, 1 \, 0 \, 1 \, 1 \, 0 \\
  17 & 3 &  4 & 1216 & 0 \, 0 \, 0 \, 0 \, 0 \, 0 \, 0 \, 0 \, 0 \, 0 \, 0 \, 1 \, 0 \, 1 \, 1 \, 1 \\
  18 & 3 &  5 & 2304 & 0 \, 0 \, 0 \, 0 \, 0 \, 0 \, 0 \, 0 \, 0 \, 0 \, 1 \, 1 \, 1 \, 1 \, 0 \, 1 \\
  19 & 3 &  6 & 2512 & 0 \, 0 \, 0 \, 0 \, 0 \, 0 \, 0 \, 0 \, 0 \, 1 \, 1 \, 0 \, 1 \, 1 \, 1 \, 1 \\
  20 & 3 &  7 & 2704 & 0 \, 0 \, 0 \, 0 \, 0 \, 0 \, 0 \, 0 \, 0 \, 1 \, 1 \, 1 \, 1 \, 1 \, 1 \, 1 \\
  21 & 3 &  8 & 1664 & 0 \, 0 \, 0 \, 0 \, 0 \, 0 \, 0 \, 1 \, 1 \, 1 \, 1 \, 0 \, 1 \, 1 \, 1 \, 1 \\
  22 & 3 &  9 &  864 & 0 \, 0 \, 0 \, 0 \, 0 \, 0 \, 1 \, 1 \, 1 \, 1 \, 0 \, 1 \, 1 \, 1 \, 1 \, 1 \\
  23 & 3 & 10 &  608 & 0 \, 0 \, 0 \, 0 \, 0 \, 1 \, 1 \, 0 \, 1 \, 1 \, 1 \, 1 \, 1 \, 1 \, 1 \, 1 \\
  24 & 3 & 11 &  384 & 0 \, 0 \, 0 \, 0 \, 0 \, 1 \, 1 \, 1 \, 1 \, 1 \, 1 \, 1 \, 1 \, 1 \, 1 \, 1 \\
  25 & 3 & 12 &  256 & 0 \, 0 \, 0 \, 1 \, 1 \, 0 \, 1 \, 1 \, 1 \, 1 \, 1 \, 1 \, 1 \, 1 \, 1 \, 1 \\
  26 & 3 & 13 &   96 & 0 \, 0 \, 0 \, 1 \, 1 \, 1 \, 1 \, 1 \, 1 \, 1 \, 1 \, 1 \, 1 \, 1 \, 1 \, 1 \\
  27 & 3 & 14 &   32 & 0 \, 0 \, 1 \, 1 \, 1 \, 1 \, 1 \, 1 \, 1 \, 1 \, 1 \, 1 \, 1 \, 1 \, 1 \, 1 \\
  \midrule
  28 & 4 &  4 &  228 & 0 \, 0 \, 0 \, 0 \, 0 \, 0 \, 0 \, 0 \, 0 \, 1 \, 1 \, 0 \, 1 \, 0 \, 0 \, 1 \\
  29 & 4 &  5 & 1648 & 0 \, 0 \, 0 \, 0 \, 0 \, 0 \, 0 \, 0 \, 0 \, 1 \, 1 \, 0 \, 1 \, 0 \, 1 \, 1 \\
  30 & 4 &  6 & 4048 & 0 \, 0 \, 0 \, 0 \, 0 \, 0 \, 0 \, 1 \, 0 \, 0 \, 1 \, 1 \, 1 \, 1 \, 1 \, 0 \\
  31 & 4 &  7 & 5856 & 0 \, 0 \, 0 \, 0 \, 0 \, 0 \, 0 \, 1 \, 0 \, 1 \, 1 \, 0 \, 1 \, 1 \, 1 \, 1 \\
  32 & 4 &  8 & 6544 & 0 \, 0 \, 0 \, 0 \, 0 \, 0 \, 0 \, 1 \, 0 \, 1 \, 1 \, 1 \, 1 \, 1 \, 1 \, 1 \\
  33 & 4 &  9 & 5104 & 0 \, 0 \, 0 \, 0 \, 0 \, 0 \, 1 \, 1 \, 0 \, 1 \, 1 \, 1 \, 1 \, 1 \, 1 \, 1 \\
  34 & 4 & 10 & 3056 & 0 \, 0 \, 0 \, 0 \, 0 \, 1 \, 1 \, 1 \, 1 \, 0 \, 1 \, 1 \, 1 \, 1 \, 1 \, 1 \\
  35 & 4 & 11 & 1504 & 0 \, 0 \, 0 \, 0 \, 1 \, 1 \, 1 \, 1 \, 1 \, 1 \, 1 \, 1 \, 0 \, 1 \, 1 \, 1 \\
  36 & 4 & 12 &  448 & 0 \, 0 \, 0 \, 1 \, 0 \, 1 \, 1 \, 1 \, 1 \, 1 \, 1 \, 1 \, 1 \, 1 \, 1 \, 1 \\
  37 & 4 & 13 &  256 & 0 \, 0 \, 1 \, 1 \, 1 \, 1 \, 0 \, 1 \, 1 \, 1 \, 1 \, 1 \, 1 \, 1 \, 1 \, 1 \\
  38 & 4 & 14 &   80 & 0 \, 1 \, 1 \, 0 \, 1 \, 1 \, 1 \, 1 \, 1 \, 1 \, 1 \, 1 \, 1 \, 1 \, 1 \, 1 \\
  39 & 4 & 15 &   16 & 0 \, 1 \, 1 \, 1 \, 1 \, 1 \, 1 \, 1 \, 1 \, 1 \, 1 \, 1 \, 1 \, 1 \, 1 \, 1 \\
  \midrule
  40 & 5 &  5 &  128 & 0 \, 0 \, 0 \, 0 \, 0 \, 0 \, 0 \, 1 \, 1 \, 0 \, 0 \, 1 \, 0 \, 1 \, 1 \, 0 \\
  41 & 5 &  6 & 1008 & 0 \, 0 \, 0 \, 0 \, 0 \, 0 \, 0 \, 1 \, 1 \, 0 \, 0 \, 1 \, 0 \, 1 \, 1 \, 1 \\
  42 & 5 &  7 & 2416 & 0 \, 0 \, 0 \, 0 \, 0 \, 0 \, 1 \, 1 \, 0 \, 1 \, 1 \, 0 \, 1 \, 1 \, 0 \, 1 \\
  43 & 5 &  8 & 3568 & 0 \, 0 \, 0 \, 0 \, 0 \, 1 \, 1 \, 0 \, 0 \, 1 \, 1 \, 0 \, 1 \, 1 \, 1 \, 1 \\
  44 & 5 &  9 & 4304 & 0 \, 0 \, 0 \, 0 \, 0 \, 1 \, 1 \, 0 \, 0 \, 1 \, 1 \, 1 \, 1 \, 1 \, 1 \, 1 \\
  45 & 5 & 10 & 3088 & 0 \, 0 \, 0 \, 0 \, 0 \, 1 \, 1 \, 1 \, 0 \, 1 \, 1 \, 1 \, 1 \, 1 \, 1 \, 1 \\
  46 & 5 & 11 & 1984 & 0 \, 0 \, 0 \, 1 \, 0 \, 1 \, 1 \, 1 \, 1 \, 1 \, 1 \, 0 \, 1 \, 1 \, 1 \, 1 \\
  47 & 5 & 12 &  904 & 0 \, 0 \, 0 \, 1 \, 1 \, 1 \, 1 \, 1 \, 1 \, 1 \, 1 \, 1 \, 0 \, 1 \, 1 \, 1 \\
  48 & 5 & 13 &  160 & 0 \, 1 \, 1 \, 0 \, 1 \, 0 \, 1 \, 1 \, 1 \, 1 \, 1 \, 1 \, 1 \, 1 \, 1 \, 1 \\
  49 & 5 & 14 &    8 & 0 \, 1 \, 1 \, 1 \, 1 \, 1 \, 1 \, 1 \, 1 \, 1 \, 1 \, 1 \, 1 \, 1 \, 1 \, 0 \\
  \midrule
  50 & 6 &  6 &   56 & 0 \, 0 \, 0 \, 0 \, 0 \, 1 \, 1 \, 0 \, 0 \, 1 \, 1 \, 0 \, 1 \, 0 \, 0 \, 1 \\
  51 & 6 &  7 &  448 & 0 \, 0 \, 0 \, 0 \, 0 \, 1 \, 1 \, 0 \, 0 \, 1 \, 1 \, 0 \, 1 \, 0 \, 1 \, 1 \\
  52 & 6 &  8 &  848 & 0 \, 0 \, 0 \, 0 \, 0 \, 1 \, 1 \, 1 \, 0 \, 1 \, 1 \, 1 \, 1 \, 0 \, 0 \, 1 \\
  53 & 6 &  9 &  928 & 0 \, 0 \, 0 \, 1 \, 0 \, 1 \, 1 \, 0 \, 0 \, 1 \, 1 \, 0 \, 1 \, 1 \, 1 \, 1 \\
  54 & 6 & 10 & 1040 & 0 \, 0 \, 0 \, 1 \, 0 \, 1 \, 1 \, 0 \, 0 \, 1 \, 1 \, 1 \, 1 \, 1 \, 1 \, 1 \\
  55 & 6 & 11 &  368 & 0 \, 0 \, 0 \, 1 \, 0 \, 1 \, 1 \, 1 \, 0 \, 1 \, 1 \, 1 \, 1 \, 1 \, 1 \, 1 \\
  56 & 6 & 12 &  172 & 0 \, 1 \, 1 \, 0 \, 1 \, 0 \, 1 \, 1 \, 1 \, 1 \, 0 \, 1 \, 1 \, 1 \, 1 \, 1 \\
  57 & 6 & 13 &   48 & 0 \, 1 \, 1 \, 0 \, 1 \, 1 \, 1 \, 1 \, 1 \, 1 \, 1 \, 1 \, 0 \, 1 \, 1 \, 1 \\
  \midrule
  58 & 7 &  7 &   16 & 0 \, 0 \, 0 \, 1 \, 0 \, 1 \, 1 \, 0 \, 0 \, 1 \, 1 \, 0 \, 1 \, 0 \, 0 \, 1 \\
  59 & 7 &  8 &  128 & 0 \, 0 \, 0 \, 1 \, 0 \, 1 \, 1 \, 0 \, 0 \, 1 \, 1 \, 0 \, 1 \, 0 \, 1 \, 1 \\
  60 & 7 &  9 &  160 & 0 \, 0 \, 0 \, 1 \, 0 \, 1 \, 1 \, 1 \, 1 \, 1 \, 1 \, 0 \, 1 \, 0 \, 0 \, 1 \\
  61 & 7 & 10 &  112 & 0 \, 0 \, 1 \, 1 \, 1 \, 1 \, 0 \, 1 \, 1 \, 1 \, 0 \, 1 \, 0 \, 1 \, 1 \, 0 \\
  62 & 7 & 11 &  128 & 0 \, 1 \, 1 \, 0 \, 1 \, 0 \, 0 \, 1 \, 1 \, 0 \, 1 \, 1 \, 1 \, 1 \, 1 \, 1 \\
  63 & 7 & 12 &   16 & 0 \, 1 \, 1 \, 0 \, 1 \, 0 \, 1 \, 1 \, 1 \, 0 \, 1 \, 1 \, 1 \, 1 \, 1 \, 1 \\
  \midrule
  64 & 8 &  8 &    2 & 0 \, 1 \, 1 \, 0 \, 1 \, 0 \, 0 \, 1 \, 1 \, 0 \, 0 \, 1 \, 0 \, 1 \, 1 \, 0 \\
  65 & 8 &  9 &   16 & 0 \, 1 \, 1 \, 0 \, 1 \, 0 \, 0 \, 1 \, 1 \, 0 \, 0 \, 1 \, 0 \, 1 \, 1 \, 1 \\
  66 & 8 & 10 &    8 & 0 \, 1 \, 1 \, 0 \, 1 \, 0 \, 1 \, 1 \, 1 \, 1 \, 0 \, 1 \, 0 \, 1 \, 1 \, 0 \\
  \bottomrule
  \end{array}
  \]
  \medskip
  \caption{Minimal representatives for $2 \times 2 \times 2 \times 2$ integer arrays}
  \label{table2222integer}
  \end{table}

%%%%%%%%%%%%%%%%%%%%%%%%%%%%%%%%%%%%%%%%%%%%%%%%%%%%%%%%%%%%%%%%%%%%%%%%

\section*{Acknowledgements}

These results form part of the Masters thesis of the second author,
written under the supervision of the first author, who was supported by a Discovery Grant from NSERC, the Natural
Sciences and Engineering Research Council of Canada.

%%%%%%%%%%%%%%%%%%%%%%%%%%%%%%%%%%%%%%%%%%%%%%%%%%%%%%%%%%%%%%%%%%%%%%%%

\end{document}